\documentclass[11pt]{amsart}
\usepackage{amsmath,amssymb,amsthm}
\usepackage{tikz}
\usetikzlibrary{arrows}
\usetikzlibrary{decorations.pathmorphing}
\usepackage[letterpaper, margin=1in]{geometry}
\usepackage{todonotes}
\usetikzlibrary{decorations.markings,intersections}
\usepackage{verbatim}
\usepackage{enumerate}
\usepackage{enumitem}
\usepackage{graphicx}

\usepackage[backref]{hyperref}
\hypersetup{
    colorlinks,
    linkcolor={red!60!black},
    citecolor={green!60!black},
    urlcolor={blue!60!black},
    pdftitle={.},
    pdfauthor={.}
}

\newcommand\abs[1]{\ensuremath{\lvert #1\rvert}}

\newcommand{\dist}{\operatorname{dist}}

\usepackage{thmtools, thm-restate}

\newtheorem{theorem}{Theorem}[section]
\newtheorem{lemma}[theorem]{Lemma}

\newtheorem{corollary}[theorem]{Corollary}
\newtheorem{proposition}[theorem]{Proposition}

\newtheorem*{claim*}{Claim}

\usepackage[utf8]{inputenc}

\title[A variant of dichromatic number for digraphs with prescribed sets of arcs]{On a variant of dichromatic number for \\ digraphs with prescribed sets of arcs}
\author{O-joung Kwon}
\author{Xiaopan Lian}

\address[O-joung Kwon]{Department of Mathematics, Hanyang University, Seoul, South Korea and Discrete Mathematics Group, Institute~for~Basic~Science~(IBS), Daejeon,~South~Korea}
\address[Xiaopan Lian]{Center for Combinatorics and LPMC, Nankai University, Tianjin, China}
\email{ojoungkwon@hanyang.ac.kr}
\email{xiaopanlian@mail.nankai.edu.cn}
\thanks{ The first author is supported by the National Research Foundation of Korea (NRF) grant funded by  the Ministry of Science and ICT (No. NRF-2021K2A9A2A11101617 and No. RS-2023-00211670) and by the Institute for Basic Science (IBS-R029-C1). }
\thanks{The second author is supported by National Natural Science Foundation of China (No. 12161141006).}
\date{\today}
\def\lrG{\overset{\textup{\hspace{0.05cm}\tiny$\leftrightarrow$}}{\phantom{\in}}\hspace{-0.32cm}G}

\def\lrU{\overset{\textup{\hspace{0.05cm}\tiny$\leftrightarrow$}}{\phantom{\in}}\hspace{-0.32cm}U}
\begin{document}
\begin{abstract}
    In this paper, we consider a variant of dichromatic number  on digraphs with prescribed sets of arcs. Let $D$ be a digraph and let $Z_1, Z_2$ be two sets of arcs in $D$. For a subdigraph $H$ of~$D$, let $A(H)$ denote the set of all arcs of $H$. Let $\mu(D, Z_1, Z_2)$ be the minimum number of parts in a vertex partition $\mathcal{P}$ of~$D$ such that for every $X\in \mathcal{P}$,  the subdigraph of~$D$ induced by $X$ contains no directed cycle $C$ with  $\abs{A(C)\cap Z_1}\neq \abs{A(C)\cap Z_2}$. For $Z_1=A(D)$ and $Z_2=\emptyset$, $\mu(D, Z_1, Z_2)$ is equal to the dichromatic number of $D$.
    
    We prove that for every digraph $F$ and every tuple $(a_e,b_e,r_e, q_e)$ of integers with $q_e\ge 2$ and $\gcd(a_e,q_e)=\gcd(b_e,q_e)=1$ for each arc $e$ of~$F$, there exists an integer $N$ such that if $\mu(D, Z_1, Z_2)\ge N$, then $D$ contains a subdigraph isomorphic to a subdivision of $F$ in which each arc $e$ of $F$ is subdivided into a directed path~$P_e$ such that~$a_e|A(P_e)\cap Z_1|+b_e|A(P_e)\cap Z_2|\equiv {r_e}\pmod {q_e}$. 
    This generalizes a theorem of Steiner [Subdivisions with congruence constraints in digraphs of large chromatic number, arXiv:2208.06358] which corresponds to the case when $(a_e, b_e, Z_1, Z_2)=(1, 1, A(D), \emptyset)$.
\end{abstract}
\maketitle

\section{Introduction}

The \emph{dichromatic number} of a digraph $D$ is the minimum integer $k$ such that the vertex set of~$D$ is partitioned into $k$ parts so that every part induces a subdigraph without directed cycles. It was introduced by Neumann-Lala~\cite{Neumann1982} in 1982, and since then, it has been considered a natural variant of chromatic number for digraphs~\cite{MoharW2016,Aboulker2019,AboulkerHKR2022,GishbolinerSS2022,TamasS2022}.

It is well known that for a fixed positive integer $t$, every graph with sufficiently large chromatic number contains a subdivision of $K_t$~\cite{BollobasT1998, KomlosS1996}. But this result does not give any additional properties on the branching paths of the obtained subdivision. Based on the fact that every graph of chromatic number at least $3$ contains an odd cycle, Toft conjectured in 1975 that every graph of chromatic number at least 4 contains a totally odd subdivision of $K_4$, which is a subdivision of $K_4$ where each branching path has odd length. Zang~\cite{Zang1998} and Thomassen~\cite{Thomassen2001} independently verified this conjecture.

Thomassen~\cite{Thomassen1983} generalized this result in the following way. For each edge~$e$ of $K_t$, a pair of integers $(r_e, q_e)$ with $q_e\ge 2$ is given. Then there is an integer $N$ such that if the chromatic number of a graph is at least $N$, then it contains a subdivision of $K_t$ where the branching path corresponding to $e$ has length $r_e\pmod {q_e}$ for every edge $e$ of $K_t$.

Aboulker et al.~\cite{Aboulker2019} initiated the study of the existence of subdivisions in digraphs of large dichromatic number, and they showed that for fixed $t$, every digraph with sufficiently large dichromatic number contains a subdivision of a complete digraph on $t$ vertices. Gir\~{a}o, Popielarz, and Snyder~\cite{GiraoPS2021} showed that every tournament with minimum out-degree at least $(2+o(1))t^2$ contains a subdivision of a complete digraph on $t$ vertices. As every digraph of dichromatic number $k$ contains an induced subdigraph with minimum out-degree at least $k-1$,  every tournament with dichromatic number at least $(2+o(1))t^2$ contains a subdivision of a complete digraph on $t$ vertices.

Steiner~\cite{Steiner2022} recently obtained an analogue of the theorem of Thomassen for digraphs.

\begin{theorem}[Steiner~\cite{Steiner2022}]\label{thm:steiner}
Let $F$ be a digraph and for every arc $e$ of $F$, let $(r_e, q_e)$ be a pair of integers with $q_e\ge 2$. Then there is an integer $N$ satisfying the following.

If the dichromatic number of a digraph $D$ is at least $N$, then $D$ contains a subdigraph isomorphic to a subdivision of $F$ in which for every arc $e$ of $F$, the branching path~$P_e$ corresponding   to $e$ satisfies that~$\abs{A(P_e)}\equiv {r_e}\pmod {q_e}$. 
\end{theorem}

In this paper, we consider a variant of dichromatic number for digraphs with prescribed sets of arcs, and generalize the theorem of Steiner. For a digraph $D$ together with a set $Z$ of arcs in $D$, it is natural to consider a minimum number of parts in a vertex partition of $D$ where each part induces a subdigraph that does not contain a directed cycle having an arc of $Z$. Cycles intersecting prescribed vertex sets or edge/arc sets have been considered in several research area. For example, algorithms for finding a minimum set to delete such cycles (Subset Feedback Vertex Set)~\cite{EvenNZ2000,CyganPPO2013,ChitnisCHM2015,HolsK2018}, and the Erd\H{o}s-P\'osa property of such cycles ($S$-cycles)~\cite{KakimuraK2012b,PontecorviW2012,HuynhJW2017,GollinHKKO2021,GollinHKOY2022} have been considered.

Motivated by the results of Thomassen~\cite{Thomassen1983} and Steiner~\cite{Steiner2022}, we prove that if this number is sufficiently large, then we can always find a subdivision of a complete digraph, where the number of arcs of $Z$ in each branching path satisfies some given modularity constraint.
We prove this in a more general setting, and for that, we introduce a new concept called $(Z_1, Z_2)$-unbalanced cycles and consider a related partition number. 

Let $D$ be a digraph and let $Z_1, Z_2$ be two sets of arcs of $D$.
For a subdigraph $H$ of $D$, let $A(H)$ denote the set of arcs in $H$. 
A directed cycle $C$ of $D$ is called \emph{$(Z_1, Z_2)$-unbalanced} if $|A(C)\cap Z_1|\neq |A(C)\cap Z_2|$.
The \emph{$(Z_1, Z_2)$-unbalanced dichromatic number} of the triple $(D, Z_1, Z_2)$, denoted by 
$\mu(D, Z_1, Z_2)$, is the minimum number of parts in a vertex partition $\mathcal{P}$ of $D$ such that for every $P\in \mathcal{P}$, the subdigraph induced by $P$ of $D$ contains no $(Z_1, Z_2)$-unbalanced directed cycles. Observe that if $Z_1=A(D)$ and $Z_2=\emptyset$, then $\mu(D_, Z_1, Z_2)$ is exactly the dichromatic number of $D$, because every directed cycle is $(Z_1, Z_2)$-unbalanced.
For convenience, when $H$ is a subdigraph of $D$, we shortly write $\mu(H, Z_1, Z_2)$ for $\mu(H, Z_1\cap A(H), Z_2\cap A(H))$.

Our main result is the following.
\begin{theorem}\label{thm:main}
Let $F$ be a digraph and for every arc $e$ of $F$, let $(a_e, b_e, r_e, q_e)$ be a tuple of integers with $q_e\ge 2$ and $\gcd(a_e, q_e)=\gcd(b_e, q_e)=1$. Then there is an integer $N$ satisfying the following.

Let $D$ be a digraph and $Z_1, Z_2\subseteq A(D)$. If $\mu(D, Z_1, Z_2)\ge N$, then $D$ contains a subdigraph isomorphic to a subdivision of $F$ in which for every arc $e$ of $F$, the branching path~$P_e$ corresponding   to $e$ satisfies that~$a_e\abs{A(P_e)\cap Z_1}+b_e\abs{A(P_e)\cap Z_2}\equiv {r_e}\pmod {q_e}$. 
\end{theorem}

By adapting $(a_e, b_e, Z_1, Z_2)=(1, 1, A(D), \emptyset)$, we get the theorem of Steiner as a corollary of Theorem~\ref{thm:main}. Also, by adapting $(a_e, b_e, Z_1, Z_2)=(1, 1, Z, \emptyset)$ for $Z\subseteq A(D)$, we get an extension of the theorem of Steiner for directed cycles containing an arc of $Z$.

We may consider an analogue of the $(Z_1, Z_2)$-unbalanced dichromatic number for undirected graphs with two prescribed edge sets, and obtain an analogous result for undirected graphs as a corollary of Theorem~\ref{thm:main}. We discuss this in Corollary~\ref{cor:undirected}.

The problem of obtaining a polynomial bound on the dichromatic number to force a subdivision of a complete digraph on $t$ vertices is a challenging problem~\cite{GishbolinerSS2022}. For undirected graphs, Bollob\'{a}s and Thomassen~\cite{BollobasT1998} and independently Koml\'{o}s and Szemer\'{e}di~\cite{KomlosS1996} showed that every undirected graph with minimum degree at least $Ct^2$ for some constant $C$ contains a subdivision of $K_t$. But such an analogue does not hold for digraphs; Thomassen~\cite{Thomassen1985} showed that there exist digraphs with arbitrarily high minimum out- and in-degree, which do not contain a subdivision of a complete digraph on $3$ vertices. So, to obtain a polynomial function, we need to develop a new method. As we mentioned, Gir\~{a}o, Popielarz, and Snyder~\cite{GiraoPS2021} obtained a polynomial function for tournaments.
We further ask whether the value $N$ in Theorem~\ref{thm:main} can be polynomial  in $\abs{V(F)}$ and $\max_{e\in A(F)}(q_e)$.

We sketch the proof of Theorem~\ref{thm:main}.
In both Theorem~\ref{thm:steiner} and Theorem~\ref{thm:main}, the main lemmas say that if $D$ has large $\mu$-value, then one can find a subset $X$ with the property that the $\mu$-value of $D[X]$ is still large, and for every pair $(u,v)$ of two vertices in $X$, there is a directed $X$-path from $u$ to $v$ satisfying the modular constraint.
 This can be achieved in an easier way for Theorem~\ref{thm:steiner}. We obtain a gadget consisting of a long directed path $P$ and a vertex $x$ where for each vertex $w$ of $P$, there is a directed $(x, V(P))$-path ending at $w$. Using such a gadget, we may control the number of arcs modulo $q_e$ in the desired path. This argument is not easily generalized to our setting, because an arc contained in both $Z_1$ and $Z_2$, or an arc contained in neither $Z_1$ nor $Z_2$, is not useful to control the values $\abs{A(P_e)\cap Z_1}$ and $\abs{A(P_e)\cap Z_2}$ individually. Also, arcs that are contained only in $Z_1$ or only in $Z_2$ might appear in an arbitrary pattern, and it is difficult to arrange them.

To overcome this difficulty, we find several gadgets, where in each gadget, there is one arc contained in $Z_1\setminus Z_2$ or $Z_2\setminus Z_1$ and we can control these gadgets independently. Then by pigeon-hole principle, we find $q_e$ gadgets where the arcs that we can control are all in $Z_1\setminus Z_2$ or all in $Z_2\setminus Z_1$. Using these arcs, we will find a desired directed path.

This paper is organized as follows. 
In Section~\ref{sec:prelim}, we introduce some preliminary concepts. 
In Section~\ref{sec:basiclemma}, we prove basic lemmas on $(Z_1, Z_2)$-unbalanced dichromatic number.
In Section~\ref{sec:main}, we prove Theorem~\ref{thm:main}, and discuss its corollaries.

\section{Preliminaries}\label{sec:prelim}

Let $\mathbb{N}$ denote the set of all positive integers.
For an integer~$m$, let~${[m]}$ denote the set of all positive integers at most~$m$. For two sets $A$ and $B$, we write $A\Delta B=(A\setminus B)\cup (B\setminus A)$.

For a digraph $D$, we
denote its vertex set by $V(D)$ and its arc set by
$A(D)$. Let $D$ be a digraph. If $e = (u, v)$ is an arc of $D$, then $u$ and $v$ are called its \emph{tail}
and \emph{head}, respectively. We say that $u$ is an \emph{in-neighbor} of $v$, and $v$ is an \emph{out-neighbor} of $u$.

For a vertex $v$ of $D$, let $D-v$ denote the digraph obtained from $D$ by removing $v$, and for a set~$S$ of vertices in $D$, let $D-S$ denote the digraph obtained by removing all vertices of $S$.
For a set~$S$ of vertices in $D$, let $D[S]$ denote the subdigraph of $D$ induced by $S$. A subdigraph $H$ of $D$ is \emph{spanning} if $V(H)=V(D)$.

For sets $A$ and $B$ of vertices of $D$, a directed path in $D$ is called a directed \emph{$(A, B)$-path} in $D$
if it starts at $A$ and ends at $B$, and all its internal vertices are not in $A\cup B$. We define $\dist_D(A, B)$ to be the length of a shortest directed $(A, B)$-path in $D$. If there is no directed path from $A$ to $B$, then we define $\dist_D(A,B):=\infty$. If $A$ or $B$ is a singleton~$\{v\}$, then we allow to replace $\{v\}$ with $v$ in the notations. For a set $A$ of vertices in $D$, a \emph{directed $A$-path}~$P$ is a directed path in $D$ such that the endvertices of $P$ are distinct vertices in $A$ and all its internal vertices are not in $A$.

 We say that $D$ is \emph{strongly connected} if for every pair of vertices $v$ and $w$ in $D$, there is a directed path from $v$ to $w$ and vice versa. A \emph{strong component} of $D$ is a maximal subdigraph of $D$ that is strongly connected.

A digraph $H$ is a \emph{subdivision} of a digraph $D$
  if $H$ can be obtained from $D$ by replacing a set $\{e_1, \dots,
  e_k\}$ of arcs in $D$ by pairwise internally vertex-disjoint directed paths $P_1, \dots, P_k$
  such that for each $i\in \{1, \ldots, k\}$, if $e_i = (u, v)$ then $P_i$ is a directed $(u, v)$-path.

A digraph $D$ is an \emph{orientation} of an undirected graph $G$, if $D$ is obtained from $G$ by giving a direction to each edge of $G$.
For a strongly connected digraph $D$ and $v\in V(D)$, a spanning subdigraph $T^+$ of $D$ is called an \emph{out-BFS-tree} of $D$ with root $v$ if 
\begin{itemize}
    \item $T^+$ is an orientation of a tree, and
    \item for every $w\in V(D)$, there is a directed $(v,w)$-path in $T^+$ and $\dist_{T^+}(v,w)=\dist_D(v, w)$.
\end{itemize}
Similarly, a spanning subdigraph $T^-$ of $D$ is called an \emph{in-BFS-tree} of $D$ with root $v$ if 
\begin{itemize}
    \item $T^-$ is an orientation of a tree, and
    \item for every $w\in V(D)$, there is a directed $(w,v)$-path in $T^-$ and $\dist_{T^-}(w,v)=\dist_D(w, v)$.
\end{itemize}
A sequence $(L_0, L_1, \ldots, L_p)$ of disjoint sets of
vertices in a strongly connected digraph $D$ is called an \emph{out-leveling} (rep. \emph{in-leveling}) in $D$ if
\begin{itemize}
    \item $\abs{L_0}=1$ and $L_0\cup L_1\cup \cdots \cup L_p=V(D)$, and 
    \item for every $i\in [p]$, every vertex in $L_i$ has at least one in-neighbor (rep. out-neighbor) in $L_{i-1}$ and has no in-neighbor (rep. out-neighbor) in $\bigcup_{j\in [i-2]}L_j$.
\end{itemize}
The vertex in $L_0$ is called the \emph{starting vertex} of $(L_0, L_1, \ldots, L_p)$, and each $L_i$ is called a \emph{level} of the leveling.
We can obtain an out-BFS-tree with root $v$ or an out-leveling with starting vertex $v$ using the BFS search with out-neighbors. 
Similarly,  
 we can also obtain an in-BFS-tree with root $v$ or an in-leveling with starting vertex $v$ using the BFS search with in-neighbors.

\section{Lemmas on $(Z_1, Z_2)$-unbalanced dichromatic number}\label{sec:basiclemma}

 In this section, we prove some necessary lemmas for proving the main theorem. 
We start with the following observation. 
\begin{lemma}\label{lem:comp}
    Let $D$ be a digraph and $Z_1, Z_2\subseteq A(D)$. Then $$\mu(D,Z_1, Z_2)=\max\{\mu(H,Z_1, Z_2): H \text{ is a strong component of }D\}.$$
\end{lemma}
\begin{proof}
    Let $D_1, \ldots, D_t$ be the set of strong components of $D$. Observe that for distinct $i,j\in [t]$, there is no directed cycle intersecting both $D_i$ and $D_j$. So, if $\mathcal{P}_i$ is a vertex partition of each $D_i$ where each part induces a subdigraph not containing a $(Z_1, Z_2)$-unbalanced directed cycle and $X_i\in \mathcal{P}_i$ for each $i\in [t]$, then $D[\bigcup_{i\in [t]}X_i]$ has no $(Z_1, Z_2)$-unbalanced directed cycle. As $\mu$ does not increase when taking a subdigraph, this shows the desired equality.
\end{proof}

In the next lemma, we show that 
if we have an in/out-leveling in a digraph with large $(Z_1, Z_2)$-unbalanced dichromatic number, then there is one level that induces a subdigraph with large $(Z_1, Z_2)$-unbalanced dichromatic number. Furthermore, in the level, we can take a strong component having large $(Z_1, Z_2)$-unbalanced dichromatic number.

\begin{lemma}\label{lem:lel}
Let $D$ be a strongly connected digraph and $Z_1, Z_2\subseteq A(D)$. Let $(L_0, L_1, \ldots,L_p)$ be an out-leveling (rep. in-leveling) of $D$. Then there are $i\in [p]\cup \{0\}$ and a strong component $H$ of $D[L_i]$  such that
$\mu(H,Z_1, Z_2)\ge \left\lceil\frac{\mu(D,Z_1, Z_2)}{2}\right\rceil$.
\end{lemma}
\begin{proof}
Let \[k=\left\lceil\frac{ \mu(D,Z_1, Z_2)}{2}\right\rceil.\] This implies $\mu(D,Z_1, Z_2)\in \{2k-1, 2k\}$. Let $D_1$ and $D_2$ be the subdigraphs of $D$ induced by the union of odd levels and the union of even levels, respectively.

We claim that there exists $i\in [p]\cup \{0\}$ such that $\mu(D[L_i],Z_1, Z_2)\ge k$.
Suppose for contradiction that for every $i\in [p]\cup\{0\}$, $\mu(D[L_i],Z_1, Z_2)\le k-1$. 
Since there is no arc from $L_{i_1}$ to $L_{i_2}$ (rep. from $L_{i_2}$ to $L_{i_1}$) when $i_2-i_1\ge 2$, it implies that $\mu(D_1,Z_1, Z_2)\le k-1$ and $\mu(D_2,Z_1, Z_2)\le k-1$. 
Since we can take different sets of colors for $D_1$ and $D_2$, we have \[\mu(D,Z_1, Z_2)\le \mu(D_1,Z_1,Z_2)+\mu(D_2,Z_1,Z_2)\le 2k-2,\] contradicting the fact that $\mu(D,Z_1,Z_2)\in \{2k-1, 2k\}$. 

Therefore, there is some $i\in [p]\cup \{0\}$ such that $\mu(D[L_i],Z_1,Z_2)\ge k$. 
By Lemma~\ref{lem:comp}, there is a strong component $H$ of $D[L_i]$ such that $\mu(H,Z_1,Z_2)=\mu(D[L_i],Z_1,Z_2)$. This proves the lemma.
\end{proof}

We prove that if $\mu(D, Z_1, Z_2)$ is large, then $D$ contains many vertex-disjoint $(Z_1, Z_2)$-unbalanced directed cycles. We can obtain it by recursively choosing a shortest $(Z_1, Z_2)$-unbalanced directed cycle, and then removing it, because a shortest $(Z_1, Z_2)$-unbalanced directed cycle always can be partitioned into two parts which do not contain $(Z_1, Z_2)$-unbalanced directed cycles.

\begin{lemma}\label{lem:disjointcycles1}
Let $t$ be a positive integer.
Let $D$ be a digraph and $Z_1, Z_2\subseteq A(D)$.
If $\mu(D,Z_1, Z_2)\ge 2t$, then $D$ contains $t$ vertex-disjoint $(Z_1, Z_2)$-unbalanced directed cycles.
\end{lemma}
\begin{proof}
We prove by induction on $t$. If $t=1$, then $\mu(D,Z_1,Z_2)\ge 2$ and $D$ contains a $(Z_1, Z_2)$-unbalanced directed cycle. So, we assume that $t>1$.

Let $C$ be a shortest $(Z_1, Z_2)$-unbalanced  directed cycle of $D$, and let $v\in V(C)$. 
If $\mu(D-V(C),Z_1,Z_2)\le 2t-3$, then $\mu(D,Z_1,Z_2)\le 2t-1$ because both $D[v]$ and $D[V(C)\setminus \{v\}]$ contains no $(Z_1, Z_2)$-unbalanced directed cycles. This contradicts the assumption that $\mu(D,Z_1,Z_2)\ge 2t$. Thus, $\mu(D-V(C), Z_1,Z_2)\ge 2t-2$.

By induction hypothesis, $D-V(C)$ contains $(t-1)$ vertex-disjoint $(Z_1, Z_2)$-unbalanced directed cycles. Thus, we obtain $t$ such directed cycles together with $C$.
\end{proof}

Now we argue that if $D$ has large $(Z_1, Z_2)$-unbalanced dichromatic number, then we can find a set $X$ of vertices in $D$ where $\mu(D[X],Z_1,Z_2)$ is still large, and for any pair $(x,y)$ of vertices in $X$, there is a directed $X$-path from $x$ to $y$. 

\begin{lemma}\label{lem:longpath}
Let $D$ be a strongly connected digraph and $Z_1,Z_2\subseteq A(D)$ with $\mu(D,Z_1,Z_2)\ge 8$. Then $D$ contains a set $X$ of vertices
such that 
\begin{itemize}
    \item $D[X]$ is strongly connected,
    \item $\mu(D[X],Z_1,Z_2)\ge \frac{\mu(D,Z_1,Z_2)}{4}$, and \item for every pair $(x,y)$ of distinct vertices in $X$, there is a directed $X$-path in $D$ from $x$ to~$y$.
\end{itemize}
\end{lemma}
\begin{proof}
Since $D$ is strongly connected, there is an in-leveling $(L_0, L_1, \ldots, L_p)$ of $D$ with starting vertex $x_0$. By Lemma~\ref{lem:lel}, there are $t\in [p]\cup \{0\}$ and $X_{1}\subseteq L_t$ such that \begin{itemize}
    \item $D[X_1]$ is a strong component of $D[L_t]$, and 
    \item $\mu(D[X_1],Z_1,Z_2)\ge \frac{\mu(D,Z_1,Z_2)}{2}\ge 4$.
\end{itemize}  
As $\mu(D[X_1],Z_1,Z_2)\ge 4$, we have $t\neq 0$.

Because $(L_0, L_1, \ldots, L_p)$ is an in-leveling of $D$, for every vertex $u\in X_1\subseteq L_t$, there is a directed path $P_{(u,x_0)}$ of length $t$ from $u$ to $x_0$, which contains exactly one vertex from each of $L_0, L_1, \ldots, L_t$. So this directed path intersects $X_1$ only at $u$.

Let $Q$ be a directed $(x_0, X_1)$-path. Such a path exists because $D$ is strongly connected. Let $\{x_1\}=V(Q)\cap X_1$. Since $D[X_1]$ is strongly connected, there is an out-leveling  $(L'_0, L'_1, \ldots, L'_{p'})$ of $D[X_1]$ with starting vertex $x_1$. Again by Lemma~\ref{lem:lel}, there are $t'\in [p']\cup \{0\}$ and $X_{2}\subseteq L'_{t'}\subseteq X_1$ such that \begin{itemize}
    \item $D[X_2]$ is a strong component of $D[L'_{t'}]$, and 
    \item $\mu(D[X_2],Z_1,Z_2)\ge \frac{\mu(D[X_1],Z_1,Z_2)}{2}\ge \frac{\mu(D,Z_1,Z_2)}{4}\ge 2$.
\end{itemize}  
As $\mu(D[X_2],Z_1,Z_2)\ge 2$, we have $t'\neq 0$.

Since $(L'_0, L'_1, \ldots, L'_{p'})$ is an out-leveling of $D[X_1]$, for every vertex $u\in X_2\subseteq L'_{t'}$, there is a directed path $P'_{(x_1,u)}$ in $D[X_1]$ from $x_1$ to $u$, which contains exactly one vertex from each of $L'_0, L'_1, \ldots, L'_{t'}$. So this directed path intersects $X_2$ only at $u$.

Set $X:=X_2$. We claim that for any pair $(x,y)$ of distinct vertices in $X$, there is a directed $X$-path in $D$ from $x$ to $y$. Note that $P_{(x,x_0)}$ is a directed $(x, x_0)$-path which intersects $X$ only at $x$, and $Q$ is a directed $(x_0, x_1)$-path which intersects $X_1$ only at $x_1$, and $P'_{(x_1,y)}$ is a directed $(x_1, y)$-path in $D[X_1]$ which intersects $X$ only at $y$. Let $R$ be a shortest $(x, x_1)$-path in $P_{(x,x_0)}\cup Q$. Then, $R\cup P'_{(x_1,y)}$ is a directed $X$-path in $D$ from $x$ to $y$. 
\end{proof}

\begin{corollary}\label{cor:longpath2}
Let $m$ be a positive integer. Let $D$ be a strongly connected digraph with $Z_1,Z_2\subseteq A(D)$. If $\mu(D,Z_1,Z_2)\ge 2^{2m+1}$, then there is a sequence $S_0\supseteq S_1\supseteq S_2\supseteq \cdots \supseteq S_m$ of sets of vertices in $D$ such that for every $i\in [m]$,
\begin{itemize}
    \item $D[S_i]$ is strongly connected, 
    \item $\mu(D[S_i],Z_1,Z_2)\ge \frac{\mu(D,Z_1,Z_2)}{4^{i}}$, and
    \item for every pair $(x,y)$ of distinct vertices in $S_m$, there is a directed $S_m$-path in $D[S_m\cup (S_{i-1}\setminus S_i)]$ from $x$ to $y$.
\end{itemize}
\end{corollary}

\begin{proof}
We claim that for every integer $0\le t\le m$,  there is a sequence of vertex subsets $V(G)=S_0\supseteq S_1\supseteq S_2\supseteq \cdots \supseteq S_t$ such that 
\begin{itemize}
    \item[(i)] for every $i\in [t]$ and  every pair $(x,y)$ of distinct 
 vertices in $S_t$, there is a directed $S_t$-path from $x$ to $y$, and 
    \item[(ii)] for every $i\in [t]$, $D[S_i]$ is strongly connected and $\mu(D[S_i],Z_1,Z_2)\ge \frac{\mu(D,Z_1,Z_2)}{4^{i}}$.
\end{itemize}
Note that ($S_0=V(D)$) is a sequence for $t=0$. 
Let $V(D)=S_0\supseteq S_1 \supseteq \cdots \supseteq S_t$ with $0\le t\le m$ be a sequence of sets of vertices in $D$ satisfying the conditions (i) and (ii) as claimed with $t$ being maximal. 
We show that $t=m$. 
Suppose that $t<m$. Note that $D[S_t]$ is a strongly connected graph with 
\[\mu(D[S_t],Z_1,Z_2)\ge \frac{\mu(D,Z_1,Z_2)}{4^{t}}\ge 2^{2m+1-2t}\ge 8.\]

We apply Lemma~\ref{lem:longpath} to $D[S_{t}]$ and obtain a set $S_{t+1}\subseteq S_{t}$ such that 
\begin{itemize}
\item $D[S_{t+1}]$ is strongly connected,  
\item $\mu(D[S_{t+1}],Z_1,Z_2)\ge \frac{\mu(D[S_t],Z_1,Z_2)}{4}\ge \frac{\mu(D,Z_1,Z_2)}{4^{t+1}}$, and
    \item for every pair $(x,y)$ of distinct vertices in $S_{t+1}$, there is a directed $S_{t+1}$-path $P$ in $G[S_t]$ from $x$ to $y$.
\end{itemize}

We have that the sequence $V(D)=S_0\supseteq S_1 \supseteq \cdots \supseteq S_t\supseteq S_{t+1}$ satisfies the conditions (i) and (ii) contradicting the fact that $t$ is maximal. Therefore, $t=m$.

By (ii), the sequence satisfies the first two required conditions. To see that it satisfies the last condition, let $i\in [m]$ and $(x,y)$ be a pair of distinct veritces in $S_m$. Since $S_m\subseteq S_i$, there is a directed $S_i$-path in $D[S_{i-1}]$ from $x$ to $y$. Since this path intersects $S_i$ only at $x$ and $y$, it is also a directed path in $D[S_m\cup (S_{i-1}\setminus S_i)]$.
\end{proof}

\section{Proof of Theorem~\ref{thm:main}}\label{sec:main}

In this section, we prove Theorem~\ref{thm:main}. In order to prove Theorem~\ref{thm:main}, we prove in Proposition~\ref{prop:Xpath} that 
if 
$\mu(D,Z_1,Z_2)$ is sufficiently large, then $D$ contains a set $X$ of vertices such that 
$\mu(D[X],Z_1,Z_2)$ is still large, and 
for every pair $(u,v)$ of distinct vertices in $X$ and every tuple $(a, b, \ell)$ of integers with $\gcd(a,q)=\gcd(b,q)=1$, there exists a directed $X$-path $Q$ from $u$ to $v$ satisfying that $a|A(Q)\cap Z_1|+b|A(Q)\cap Z_2|\equiv \ell\pmod q$.
Using this proposition, we can show Theorem~\ref{thm:main} by the induction on the number of arcs of $F$.

We obtain Proposition~\ref{prop:Xpath}
 as follows. We illustrate in Figure~\ref{fig:Xpath}. By Lemma~\ref{lem:comp}, we may assume that $D$ is strongly connected. First, we take an in-leveling of $D$ with some starting vertex $x$. By Lemma~\ref{lem:lel}, there is a set $X^*$ of vertices such that $D[X^*]$ is strongly connected and $\mu(D[X^*], Z_1, Z_2)$ is large.
 Let $P$ be a directed $(x, X^*)$-path and let $x^*$ be the last vertex of $P$. From $X^*$ with $x^*$, we recursively take an out-leveling, and then take a strongly connected subdigraph with large $\mu$-value. In each step, we keep one particular directed path, which contains an arc that can be used to control the value $\abs{A(P_e)\cap Z_1}$ or $\abs{A(P_e)\cap Z_2}$. At the end, we obtain a desired set $X$.  When we choose two vertices $u$ and $v$, as $X^*$ was chosen from an in-leveling, there is a directed $(u, x)$-path in $D[(V(D)\setminus X^*)\cup \{u\}]$. Controlling arcs that are contained in $Z_1\Delta Z_2$ that we find in $X^*$, we will obtain a directed $X$-path from $u$ to $v$, with $a|A(Q)\cap Z_1|+b|A(Q)\cap Z_2|\equiv \ell\pmod q$.

   \begin{figure}
  \centering
  \begin{tikzpicture}[scale=0.8]
  \tikzstyle{w}=[circle,draw,fill=black!50,inner sep=0pt,minimum width=4pt]
  \tikzset{edge/.style = {->,> = latex'}}

    \draw[rounded corners=2mm, thick] (-2, 5)--(7, 5)--(7,0)--(-7,0)--(-7,5)--(0,5);

    \draw[rounded corners=2mm, thick] (-2, 4.7)--(2, 4.7)--(2,1.3)--(-6.7,1.3)--(-6.7,4.7)--(-2,4.7);

    \draw[rounded corners=2mm, thick] (-3, 4.4)--(-2, 4.4)--(-2,1.6)--(-6.4,1.6)--(-6.4,4.4)--(-3,4.4);

     \node at (4.5, 0.5) {$D$};
     \node at (1.5, 1.8) {$X^*$};
     \node at (-2.5, 2) {$X$};
    
    \draw (4, 3) node [w] (x) {};
     \node at (4.5, 3) {$x$};

    \draw (-3, 3) node [w] (w4) {};          \node at (-3.5, 3) {$v$};

    \draw (-4, 2) node [w] (w5) {};       
    \draw (-2, 0.5) node [w] (w6) {};       
    \draw (0, 0.5) node [w] (w7) {};       
    \draw (2, 0.5) node [w] (w8) {};          
    \node at (-4.5, 2) {$u$};
       \draw[edge, thick] (w5) to (w6);
       \draw[edge, thick] (w6) to (w7);
       \draw[edge, thick] (w7) to (w8);
       \draw[edge, thick] (w8) to (x);

        \draw (1.5, 3) node [w] (w2) {};        
     \node at (1.5, 3.5) {$x^*$};
     \node at (3, 3.3) {$P$};

       \draw[edge, thick, dashed] (w2) to (w4);

       \draw[edge, thick] (x) to (w2);

  \end{tikzpicture}     \caption{An illustration of the sketch of Proposition~\ref{prop:Xpath}. }\label{fig:Xpath}
\end{figure}
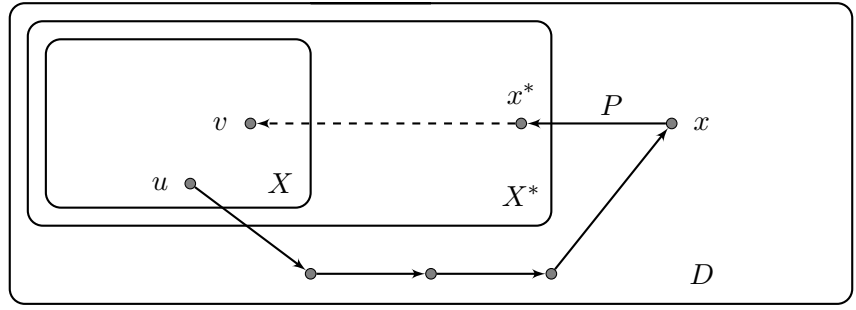

We start with the following lemma.
\begin{lemma}\label{lem:cycletwoedges}
    Let $D$ be a strongly connected digraph and $Z_1, Z_2\subseteq A(D)$.
    If $\mu(D, Z_1, Z_2)\ge 1536$, then $D$ contains a directed cycle containing at least two arcs in $Z_1\Delta Z_2$.
\end{lemma}
\begin{proof}
By Corollary~\ref{cor:longpath2} with $m=4$, there is a sequence $S_0\supseteq S_1\supseteq \cdots \supseteq S_4$ of sets of vertices in~$D$ such that for every $i\in [4]$,
\begin{itemize}
        \item $D[S_i]$ is strongly connected,
        \item $\mu(D[S_i],Z_1,Z_2)\ge \frac{\mu(D,Z_1,Z_2)}{4^{i}}$, and
        \item for every pair $(x,y)$ of distinct vertices in $S_4$, there is a directed $S_4$-path in $D[S_4\cup (S_{i-1}\setminus S_i)]$ from $x$ to $y$.
\end{itemize}
For each $j\in [4]$, let $L_j:=S_{j-1}\setminus S_j$. Then $L_i\cap L_j=\emptyset$ for distinct integers $i,j\in [4]$.

Observe that
\[\mu(D[S_4],Z_1,Z_2)\ge \frac{\mu(D,Z_1,Z_2)}{4^4}\ge 6.\]
Let $v_1$ and $v_2$ be two vertices in $S_4$ and let $T:=S_4\setminus \{v_1,v_2\}$. Since $\mu(D[S_4],Z_1,Z_2)\ge 6$, we have that  
    \[\mu(D[T],Z_1,Z_2)\ge \mu(D[S_4],Z_1,Z_2)-|\{v_1,v_2\}|\ge 4.\] 
Therefore, by Lemma~\ref{lem:disjointcycles1} with $t=2$, $D[T]$ contains two vertex-disjoint $(Z_1,Z_2)$-unbalanced directed cycles, say $C_1$ and $C_2$. By the definition of $(Z_1,Z_2)$-unbalanced directed cycles, for each $i\in [2]$, there is an arc $e_i=(x_i, y_i)$ in $C_i$ such that $e_i\in Z_1\Delta Z_2$.

Note that $\{v_1,v_2,x_1,x_2,y_1,y_2\}\subseteq S_4$. By the third property of the sequence $S_0\supseteq \cdots \supseteq S_4$, there are 
\begin{itemize}
    \item a directed $(v_1, x_1)$-path $P_{(v_1,x_1)}$ in $D[\{v_1,x_1\}\cup L_1]$, 
    \item a directed $(y_1, v_2)$-path $P_{(y_1,v_2)}$ in $D[\{y_1,v_2\}\cup L_2]$, 
    \item a directed $(v_2, x_2)$-path $P_{(v_2,x_2)}$ in $D[\{v_2,x_2\}\cup L_3$], and 
    \item a directed $(y_2, x_1)$-path $P_{(y_2,x_1)}$ in $D[\{v_2,x_1\}\cup L_4]$. 
\end{itemize}
Let \[C:=P_{(v_1,x_1)}\cup D[\{x_1,y_1\}]\cup P_{(y_1,v_2)}\cup P_{(v_2,x_2)}\cup D[\{x_2,y_2\}]\cup P_{(y_2,x_1)}.\] 
Since $L_i\cap L_j=\emptyset$ for distinct integers $i,j\in [4]$, $C$ is a directed cycle. As $e_1$ and $e_2$ are contained in exactly one of $Z_1$ and $Z_2$, $C$ is a desired directed cycle.
\end{proof}

    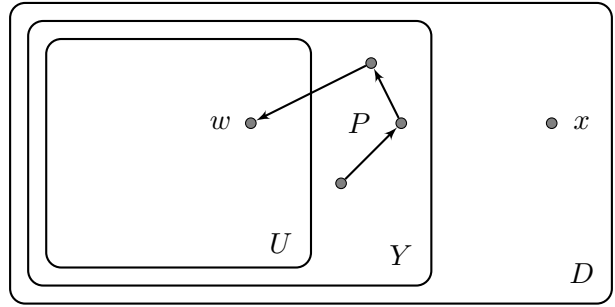
\begin{figure}
  \centering
  \begin{tikzpicture}[scale=0.8]
  \tikzstyle{w}=[circle,draw,fill=black!50,inner sep=0pt,minimum width=4pt]
  \tikzset{edge/.style = {->,> = latex'}}

    \draw[rounded corners=2mm, thick] (0, 5)--(5, 5)--(5,0)--(-5,0)--(-5,5)--(0,5);

    \draw[rounded corners=2mm, thick] (-2, 4.7)--(2, 4.7)--(2,0.3)--(-4.7,0.3)--(-4.7,4.7)--(-2,4.7);

    \draw[rounded corners=2mm, thick] (-2, 4.4)--(0, 4.4)--(0,0.6)--(-4.4,0.6)--(-4.4,4.4)--(-2,4.4);

     \node at (4.5, 0.5) {$D$};
     \node at (1.5, 0.8) {$Y$};
     \node at (-0.5, 1) {$U$};
    \node at (0.8, 3) {$P$};
    
    \draw (4, 3) node [w] (x) {};
     \node at (4.5, 3) {$x$};

    \draw (-1, 3) node [w] (w4) {};          \node at (-1.5, 3) {$w$};
    
        \draw (0.5, 2) node [w] (w1) {};        \draw (1.5, 3) node [w] (w2) {};        \draw (1, 4) node [w] (w3) {};            
       \draw[edge, thick] (w1) to (w2);
       \draw[edge, thick] (w2) to (w3);
       \draw[edge, thick] (w3) to (w4);

  \end{tikzpicture}     \caption{The sets $U$ and $Y$ and $w\in U$ and the directed path $P$ obtained in Lemma~\ref{lem:specialset}. The first arc of $P$ is contained in $Z_1\Delta Z_2$. }\label{fig:twosubsets}
\end{figure}

The following lemma is used to recursively find a subset of large $\mu$-value together with some particular path containing an arc that is useful to control the values $\abs{A(P_e)\cap Z_1}$ and $\abs{A(P_e)\cap Z_2}$ later.  We illustrate the sets and the path obtained in Lemma~\ref{lem:specialset} in Figure~\ref{fig:twosubsets}. 
\begin{lemma}\label{lem:specialset}
Let $q\ge 2$ be an integer. Let $D$ be a strongly connected digraph with $x\in V(D)$ and let $Z_1,Z_2\subseteq A(D)$. If 
$\mu(D,Z_1,Z_2)\ge 3072q^2$, then there exist two vertex subsets $U\subseteq Y\subseteq V(D)\setminus \{x\}$, a vertex $w\in U$, a directed path $P$ in $D[Y]$ of length at least $1$, and $r,s\in [q]$  such that 
\begin{enumerate}
    \item $D[U]$ and $D[Y]$ are strongly connected,
    
 \item $V(P)\cap U=\{w\}$ and $w$ is the last vertex of $P$,

\item $\mu(D[U], Z_1, Z_2)\ge \frac{\mu(D,Z_1,Z_2)}{2}-1536$, 
    \item the first arc of $P$ is in $Z_1\Delta Z_2$, and 
    \item for the first two vertices $v$ of $P$, there is an $(x, v)$-path $Q$ in $D[(V(D)\setminus Y)\cup \{v\}]$ satisfying that 
$|A(Q)\cap Z_1|\equiv r\pmod q$ and $|A(Q)\cap Z_2|\equiv s\pmod q$.
\end{enumerate}
\end{lemma}
\begin{proof}
Since $D$ is strongly connected, there is an out-BFS-tree $T^+$ of $D$ with root $x$. Let $L_{j}=\{v\in V(D):\dist_{T^+}(x,v)=j\}$. Let $p$ be the maximum distance from $x$ to a vertex in $D$. Then $(L_0,L_1,\ldots, L_{p})$ is an out-leveling of $D$ with starting vertex $x$. 
By Lemma~\ref{lem:lel} there are $i\in [p]\cup \{0\}$  and $Y\subseteq L_{i}$ such that  

\begin{itemize}
\item $D[Y]$ is a strong component of $D[L_{i}]$, and 
\item $\mu(D[Y],Z_1,Z_2)\ge \frac{\mu(D,Z_1,Z_2)}{2}$. 
\end{itemize}
As $\mu(D[Y],Z_1,Z_2)\ge 2$, we have $i\neq 0$.

Since $T^+$ is an out-BFS-tree of $D$, for each $v\in Y$, there is a unique directed $(x, v)$-path $Q_v$ in $T^+$ that intersects $Y$ only at $v$. 
 We define an equivalence relation $\sim$ on $Y$ such that for vertices $u,v\in Y$, $u\sim v$ if and only if \begin{itemize}
     \item $|A(Q_u)\cap Z_1|\equiv |A(Q_v)\cap Z_1|\pmod q$ and
     \item $|A(Q_u)\cap Z_2|\equiv |A(Q_v)\cap Z_2|\pmod q$.
\end{itemize}
Observe that there are at most $q^2$ equivalence classes. 
Thus, there is an equivalence class, say $Y^*$, satisfying that  
\begin{itemize}
    \item $\mu(D[Y^*],Z_1,Z_2)\ge \frac{\mu(D[Y],Z_1,Z_2)}{q^2}\ge 1536$.
\end{itemize}
Let $r, s\in [q]$ with $r\equiv |A(Q_u)\cap Z_1|\pmod q$ and
     $s\equiv |A(Q_u)\cap Z_2|\pmod q$ for $u\in Y^*$.

Let $Y^{**}\subseteq Y^*$ be a minimal subset with $\mu(D[Y^{**}], Z_1, Z_2)\ge 1536$. Then the minimality of~$Y^{**}$ and Lemma~\ref{lem:comp} imply that $D[Y^{**}]$ is strongly connected and  $\mu(D[Y^{**}], Z_1, Z_2)= 1536$. As $\mu(D[Y^{**}], Z_1, Z_2)= 1536$, by Lemma~\ref{lem:cycletwoedges}, $D[Y^{**}]$ contains a directed cycle containing at least two arcs in $Z_1\Delta Z_2$. We call it $C$.

Let $U\subseteq Y\setminus Y^{**}$ be a minimal subset with $\mu(D[U], Z_1, Z_2)= \mu(D[Y\setminus Y^{**}], Z_1, Z_2)$.  Then $D[U]$ is strongly connected. Since $\mu(D[Y],Z_1,Z_2)\ge  \frac{\mu(D,Z_1,Z_2)}{2}$,  we have
\begin{align*}
\mu(D[U], Z_1, Z_2)&\ge \mu(D[Y], Z_1, Z_2)-\mu(D[Y^{**}], Z_1, Z_2) \\
&\ge   \frac{\mu(D,Z_1,Z_2)}{2}-1536.
\end{align*}
Let $P^*$ be a directed $(V(C), U)$-path in $D[Y]$. Such a path exists as $D[Y]$ is strongly connected. As $C$ contains two arcs in $Z_1\Delta Z_2$, $C\cup P^*$ contains a subpath $P$ ending at $U$ such that the starting arc is an arc of $C$ contained in $Z_1\Delta Z_2$. Denote by  $w$ the last vertex of $P$ in $U$. Observe that $Y$, $U$, $x$, $w$, $P$, $r$, $s$ satisfy the conditions (1)--(5).
\end{proof}

Lemma~\ref{lem:sequences} describes sequences of sets and paths that can be obtained by applying Lemma~\ref{lem:specialset} recursively.
\begin{lemma}\label{lem:sequences}
Let $q\ge 2$ be an integer. Let $D$ be a strongly connected digraph with $x\in V(D)$ and let $Z_1,Z_2\subseteq A(D)$. If 
$\mu(D,Z_1,Z_2)\ge 1536\cdot 2^{2q-3}(q^2+1)-3072$, then there are sequences $(X_j:j\in [2q-2])$, $(Y_j:j\in [2q-3])$ of sets of vertices in $D$  and a sequence $(x_j:j\in [2q-2])$ of vertices in $D$ and a sequence $(P_j:j\in [2q-3])$ of directed paths in $D$  and sequences $(r_j:j\in [2q-3])$, $(s_j:j\in [2q-3])$  of integers such that $X_1=V(D)$,  $D[X_{2q-2}]$ is strongly connected, and for every $i\in [2q-3]$, 
\begin{enumerate}
    \item $x_i\in X_i$ and $X_{i+1}\subseteq Y_i\subseteq X_i\setminus \{x_i\}$,
    \item $D[X_i]$ and $D[Y_i]$ are strongly connected,
    
    \item $P_i$ is a directed path in $D[Y_i]$ of length at least $1$, $V(P_i)\cap X_{i+1}=\{x_{i+1}\}$ and $x_{i+1}$ is the last vertex of $P_i$,

    \item $\mu(D[X_{i+1}],Z_1,Z_2)\ge \frac{\mu(D[X_{i}],Z_1,Z_2)}{2}-1536$,
    \item the first arc of $P_i$ is in $Z_1\Delta Z_2$, and
    
    \item for the first two vertices $v$ of $P_i$, there is an $(x_i, v)$-path $Q_{(x_i,v)}$ in $D[(X_i\setminus Y_i)\cup \{v\}]$ satisfying that 
$|A(Q_{(x,v)})\cap Z_1|\equiv r_i\pmod q$ and $|A(Q_{(x,v)})\cap Z_2|\equiv s_i\pmod q$.
 \end{enumerate}    
 In particular, if 
 $\mu(D,Z_1,Z_2)\ge 2^{2q-3} \cdot\max (1536(q^2+1), N+3072)-3072$ for some positive integer~$N$, then 
 $\mu(D[X_{2q-2}, Z_1, Z_2])\ge N$.
\end{lemma}
\begin{proof}
Suppose that $\mu(D,Z_1,Z_2)\ge 1536\cdot 2^{2q-3}(q^2+1)-3072$.
We claim that for every $m\in [2q-3]$, there are sequences $(X_j:j\in [m+1])$, $(Y_j:j\in [m])$, $(x_j:j\in [m+1])$, $(P_j:j\in [m])$, $(r_j:j\in [m])$, and $(s_j:j\in [m])$ satisfying the conditions (1)--(6). 

As $q\ge 2$, we have $\mu(D,Z_1,Z_2)\ge 1536\cdot 2(q^2+1)-3072\ge 3072q^2$. So, by applying Lemma~\ref{lem:specialset} for $(D, x)$,  
there exist two vertex subsets $U\subseteq Y\subseteq V(D)\setminus \{x\}$, a vertex $w\in U$, a directed path $P$ in $D[Y]$ of length at least $1$, and $r,s\in [q]$  such that 
\begin{enumerate}
    \item $D[U]$ and $D[Y]$ are strongly connected,
    
 \item $V(P)\cap U=\{w\}$ and $w$ is the last vertex of $P$,

\item $\mu(D[U], Z_1, Z_2)\ge \frac{\mu(D,Z_1,Z_2)}{2}-1536$, 
    \item the first arc of $P$ is in $Z_1\Delta Z_2$, and 
    \item for the first two vertices $v$ of $P$, there is an $(x, v)$-path $Q$ in $D[(V(D)\setminus Y)\cup \{v\}]$ satisfying that 
$|A(Q)\cap Z_1|\equiv r\pmod q$ and $|A(Q)\cap Z_2|\equiv s\pmod q$.
\end{enumerate}
We consider $X_1:=V(D)$, $x_1:=x$, and we set $Y_1:=Y$, $X_2:=U$, $P_1:=P$, $x_2:=w$, $r_1:=r$ and $s_1:=s$.
Then  
$(X_1,X_2)$, $(Y_1)$, $(x_1,x_2)$, $(P_1)$, $(r_1)$ and $(s_1)$ are sequences for $m=1$ satisfying the conditions (1)--(6).

Now, let $m$ be the maximum integer in $[2q-3]$ such that there are sequences $(X_j:j\in [m+1])$, $(Y_j:j\in [m])$, $(x_j:j\in [m+1])$, $(P_j:j\in [m])$, $(r_j:j\in [m])$ and $(s_j:j\in [m])$ satisfying the conditions (1)--(6). We show that $m=2q-3$. Suppose that $m<2q-3$.

By the construction, $D[X_{m+1}]$ is  strongly connected. Note that for each $i\in [m]$, we have
\[ 2\Big( \mu(D[X_{i+1}],Z_1,Z_2)+3072 \Big)\ge \mu(D[X_{i}],Z_1,Z_2)+3072. \]
It implies that 
$$\begin{aligned}
\mu(D[X_{m+1}],Z_1,Z_2)&\ge \frac{1}{2^{m}}\Big(\mu(D,Z_1,Z_2)+3072\Big) -3072 \\
&\ge 1536\cdot 2(q^2+1) -3072 = 3072q^2.
\end{aligned}$$

 By applying Lemma~\ref{lem:specialset} for $ (D[X_{m+1}], x_{m+1})$, 
 there exist two vertex subsets $U^*\subseteq Y^*\subseteq X_{m+1}\setminus \{x_{m+1}\}$, a vertex $w^*\in U^*$, a directed path $P^*$ in $D[Y^*]$ of length at least $1$, and $r^*,s^*\in [q]$  such that 
\begin{enumerate}
    \item $D[U^*]$ and $D[Y^*]$ are strongly connected,
    
 \item $V(P^*)\cap U^*=\{w^*\}$ and $w^*$ is the last vertex of $P^*$,

\item $\mu(D[U^*], Z_1, Z_2)\ge \frac{\mu(D[X_{m+1}],Z_1,Z_2)}{2}-1536$, 
    \item the first arc of $P^*$ is in $Z_1\Delta Z_2$, and 
    \item for the first two vertices $v$ of $P^*$, there is an $(x_{m+1}, v)$-path $Q$ in $D[(X_{m+1}\setminus Y)\cup \{v\}]$ satisfying that 
$|A(Q)\cap Z_1|\equiv r\pmod q$ and $|A(Q)\cap Z_2|\equiv s\pmod q$.
\end{enumerate}
 We set $Y_{m+1}:=Y^*$, $X_{m+2}:=U^*$,  $P_{m+1}:=P^*$, $x_{m+2}:=w^*$, 
$r_{i+1}:=r^*$, $s_{i+1}:=s^*$, respectively. Then 
$(X_j:j\in [m+2])$, $(Y_j:j\in [m+1])$, $(x_j:j\in [m+2])$, $(P_j:j\in [m+1])$, $(r_j:j\in [m+1])$ and $(s_j:j\in [m+1])$ satisfy the conditions (1)--(6) which contradicts the maximality of $m$.

Therefore, $m=2q-3$.
By the recurrence relation, we have 
\[\mu(D[X_{2q-2}],Z_1,Z_2)\ge \frac{1}{2^{2q-3}}\Big(\mu(D,Z_1,Z_2)+3072\Big) -3072
\ge (N+3072) -3072 = N.\]
This proves the theorem. 
\end{proof}

\begin{proposition}\label{prop:Xpath}
There is a function $g:\mathbb{N}\times \mathbb{N}\to \mathbb{N}$ satisfying the following. 

Let $q\ge 2$ and $N\ge 2$ be integers. Let $D$ be a strongly connected digraph and let $Z_1,Z_2\subseteq A(D)$. If 
$\mu(D,Z_1,Z_2)\ge g(q,N)$, then there exists a set $X\subseteq V(D)$ such that 
\begin{itemize}
\item $\mu(D[X],Z_1,Z_2)\ge N$, and 
\item for every pair $(u,v)$ of distinct vertices in $X$ and every tuple $(a, b, \ell)$ of integers with $\gcd(a,q)=\gcd(b,q)=1$, there exists a directed $X$-path $Q$ from $u$ to $v$ satisfying that $a|A(Q)\cap Z_1|+b|A(Q)\cap Z_2|\equiv \ell\pmod q$.
\end{itemize}
\end{proposition}
\begin{proof}
Set
\begin{itemize}
    \item $g(q,N):=2^{2q-2}\cdot \max (1536 (q^2+1), 2N+3072)-6144$.
\end{itemize}

Since $D$ is strongly connected, there is an in-leveling $(L_0, L_1, \ldots,L_p)$ of $D$ with starting vertex~$x_0$. By Lemma~\ref{lem:lel}, there are $\delta\in [p]\cup \{0\}$ and a set $X^*\subseteq V(D)$ such that 
\begin{itemize}
\item $D[X^*]$ is a strong component of $D[L_\delta]$, and 
\item $\mu(D[X^*],Z_1,Z_2)\ge \frac{g(q,N)}{2}$. %
\end{itemize}
As $\mu(D[X^*],Z_1,Z_2)\ge 2$, we have $\delta\neq 0$. Let $P$ be a directed  $(x_0, X^*)$-path in $D$, and let $x_1$ be the last vertex of $P$. For every $u\in X^*$, there is a directed path $P_{(u,x_0)}$ in $D$ from $u$ to $x_0$ intersecting $X^*$ at $u$ in $D$, which passes through exactly one vertex in $L_j$ for each $j\in [\delta]$. For each $u\in X^*$, let $P_{(u,x_1)}$ be a directed path from $u$ to $x_1$ contained in $P\cup P_{(u,x_0)}$. Then $V(P_{(u, x_1)})\cap X^*=\{x_1,u\}$.

Observe that  
\[\mu(D[X^*],Z_1,Z_2)\ge \frac{g(q,N)}{2}= 2^{2q-3}\cdot \max (1536 (q^2+1), 2N+3072)-3072.\] By applying Lemma~\ref{lem:sequences} for $(X^*, x_1)$, there are sequences $(X_j:j\in [2q-2])$, $(Y_j:j\in [2q-3])$ of sets of vertices in $D$  and a sequence $(x_j:j\in [2q-2])$ of vertices in $D$ and a sequence $(P_j:j\in [2q-3])$ of directed paths in $D$  and sequences $(r_j:j\in [2q-3])$, $(s_j:j\in [2q-3])$  of integers such that $X_1=X^*$, $D[X_{2q-2}]$ is strongly connected, $\mu(D[X_{2q-2}], Z_1, Z_2)\ge 2N$ and for every $i\in [2q-3]$, 
\begin{enumerate}
    \item $x_i\in X_i$ and $X_{i+1}\subseteq Y_i\subseteq X_i\setminus \{x_i\}$, 
    \item $D[X_i]$ and $D[Y_i]$ are strongly connected,
    
    \item $P_i$ is a directed path in $D[Y_i]$ of length at least $1$, $V(P_i)\cap X_{i+1}=\{x_{i+1}\}$ and $x_{i+1}$ is the last vertex of $P_i$,

    \item $\mu(D[X_{i+1}],Z_1,Z_2)\ge \frac{\mu(D[X_i],Z_1,Z_2)}{2}-1536$, 
    \item the first arc of $P_i$ is in $Z_1\Delta Z_2$,  and
    
    \item for the first two vertices $v$ of $P_i$, there is an $(x_i, v)$-path $Q_{(x_i,v)}$ in $D[(X_i\setminus Y_i)\cup \{v\}]$ satisfying that 
$|A(Q_{(x_i,v)})\cap Z_1|\equiv r_i\pmod q$ and $|A(Q_{(x_i,v)})\cap Z_2|\equiv s_i\pmod q$.
 \end{enumerate}
We take an out-leveling $(L'_0,L'_1,\ldots, L'_{t} )$ of $D[X_{2q-2}]$ with starting vertex $x_{2q-2}$. By Lemma~\ref{lem:lel} there are $\delta'\in [t]\cup \{0\}$  and $X\subseteq L'_{\delta'}$ such that  

\begin{itemize}
\item $D[X]$ is a strong component of $D[L'_{\delta'}]$, and 
\item $\mu(D[X],Z_1,Z_2)\ge \frac{\mu(D[X_{2q-2}],Z_1,Z_2)}{2}\ge N$. 
\end{itemize}
As $\mu(D[X],Z_1,Z_2)\ge N\ge 2$, we have $\delta'\neq 0$. Observe that for every $u\in X$, there is a directed path $P_{(x_{2q-2},u)}$ from $x_{2q-2}$ to $u$ in $D[X_{2q-2}]$ intersecting $X$ at $u$ which passes through exactly one vertex in $L'_j$ for each $j\in [\delta']\cup \{0\}$.

We claim that this set $X$ is the desired set. To see this, 
it suffices to show that for every pair of distinct vertices $u,v\in X$ and every tuple $(a, b, \ell)$ of integers with $\gcd(a,q)=\gcd(b,q)=1$, there exists a directed $X$-path $Q$ from $u$ to $v$ satisfying that $a|A(Q)\cap Z_1|+b|A(Q)\cap Z_2|\equiv \ell\pmod q$.
Let $u,v$ be distinct vertices of $X$ and let $(a, b, \ell)$ be a tuple of integers with $\gcd(a,q)=\gcd(b,q)=1$.

 For each $j\in [2q-3]$, let $u_j$ and $v_j$ be the first and second vertices of $P_j$, respectively. By Condition~(5), the arc $(u_j, v_j)$ is contained in $Z_1\Delta Z_2$.
By the pigeonhole principle, $\{(u_j,v_j):j\in [2q-3]\}$ contains either at least $q-1$ arcs in $Z_1\setminus Z_2$ or at least $q-1$ arcs in $Z_2\setminus Z_1$. 
We assume that $\{(u_j,v_j):j\in [2q-3]\}$ contains at least $q-1$ arcs in $Z_1\setminus Z_2$. When $\{(u_j,v_j):j\in [2q-3]\}$ contains at least $q-1$ arcs in $Z_2\setminus Z_1$, the proof will be symmetric.

Let $I\subseteq [2q-3]$ be a set of size exactly $q-1$ where the arcs in $\{(u_j, v_j):j\in I\}$ are contained in $Z_1\setminus Z_2$.
Let $i_1<i_2< \cdots <i_{q-1}$ be the integers in $I$.
For each $k\in [q]$,  let $I_k:=\{i_j:j\in [k-1]\}$.

For each $k\in [q]$, let \[Q_k=P_{(u,x_1)} \cup P_{(x_{2q-2},v)}\cup \bigcup_{j\in I_k}(Q_{(x_j,u_j)}\cup P_j)\cup \bigcup_{j\in [2q-3]\setminus I_k} (Q_{(x_j, v_j)}\cup (P_j-u_j)).\]

By Condition~(6), for each $i\in [2q-3]$ and for $v\in \{u_i,v_i\}$, the directed path $Q_{(x_i,v)}$ in $D[(X_i\setminus Y_i)\cup \{v\}]$ satisfies that 
$|A(Q_{(x_i,v)})\cap Z_1|\equiv r_i\pmod q$ and $|A(Q_{(x_i,v)})\cap Z_2|\equiv s_i\pmod q$. Since $(u_i,v_i)\in Z_1\setminus Z_2$ for each $i\in I$, we have that for each $k\in [q]$, $$
\begin{aligned}|A(Q_k)\cap Z_1|&\equiv k-1+|A(Q_1)\cap Z_1|\pmod q,\\
|A(Q_k)\cap Z_2|&\equiv |A(Q_1)\cap Z_2|\pmod q.
\end{aligned}$$
Thus, for distinct $i,j\in [q]$, we have
\[|A(Q_i)\cap Z_1|\not\equiv |A(Q_j)\cap Z_1|{\pmod q} ~\quad\text{and}~\quad |A(Q_i)\cap Z_2|\equiv |A(Q_j)\cap Z_2| {\pmod q}.
\]

Since $\gcd(a, q)=1$, 
the congruence \[ar\equiv \ell-b|A(Q_1)\cap Z_2| \pmod q\] has a unique solution for $r$ in $[q]$.
Let $s\in [q]$ such that $s\equiv r-|A(Q_1)\cap Z_1|+1\pmod q$.
Then $Q_s$ is a directed $X$-path from $u$ to $v$ satisfying that 
$$\begin{aligned}
a|A(Q_s)\cap Z_1|+b|A(Q_s)\cap Z_2|&\equiv  a(|A(Q_1)\cap Z_1|+(s-1))+b|A(Q_s)\cap Z_2| \\ 
&\equiv ar+b|A(Q_s)\cap Z_2|  
\\
&\equiv \ell\pmod q
\end{aligned}$$
This proves the proposition.
\end{proof}

 Now, we prove our main result.
 
\begin{theorem}\label{thm:mainthm}
Let $F$ be a digraph, and for every $e\in A(F)$ let $(a_e,b_e,r_e,q_e)$ be a tuple of integers with $q_e \ge 2$ and $\gcd(a_e,q_e)=\gcd(b_e,q_e)=1$.
Then there exists an integer $N$ satisfying  the following. 

Let $D$ be a digraph and $Z_1, Z_2\subseteq A(D)$. If $\mu(D,Z_1,Z_2)
\ge N$, then $D$ contains a
subdivision $H$ of $F$ such that for every $e\in  A(F)$, the corresponding  branching path $P_e$ in $H$ satisfies that $a_e|A(P_e)\cap Z_1|+b_e|A(P_e)\cap Z_2|\equiv r_e \pmod {q_e}$.
\end{theorem}
\begin{proof}
Let $g$ be the function defined in Proposition~\ref{prop:Xpath}. 

We prove the theorem by induction on $\abs{A(F)}$. For $\abs{A(F)}=0$, the statement holds with $N=|V(F)|$. Suppose that $\abs{A(F)} \ge 1$ and suppose that the theorem holds for any digraph $F'$ where $\abs{A(F')}<\abs{A(F)}$. 
 Let $f\in A(F)$ and 
let $N'$ be the integer obtained by applying induction to $F-f$.

Let $N:= g(q_f, N')$. Let $D$ be a strongly connected digraph with $Z_1,Z_2\subseteq A(D)$ such that $\mu(D,Z_1,Z_2)\ge N$. By Proposition~\ref{prop:Xpath}, there exists a set $X\subseteq V(D)$ such that 
\begin{enumerate}
\item $\mu(D[X],Z_1,Z_2)\ge N'$, and 
\item for every pair of vertices $u,v\in X$ and every tuple $(a, b, \ell)$ of integers with $\gcd(a,q)=\gcd(b,q)=1$, there exists a directed $X$-path $Q$ from $u$ to $v$ satisfying that $a|A(Q)\cap Z_1|+b|A(Q)\cap Z_2|\equiv \ell\pmod q$.
\end{enumerate}
Since $\mu(D[X],Z_1,Z_2)\ge N'$, by the induction hypothesis, $D[X]$ contains a subdivision $H'$ of $F'$ where for every $e\in  A(F')$, the corresponding branching path $P_e$ of $H'$ satisfies that $a_e|A(P_e)\cap Z_1^*|+b_e|A(P_e)\cap Z_2^*|\equiv r_e \pmod {q_e}$. Let $u$ and $v$ be the branching vertices in $H'$ which correspond to the head and tail vertices of $f$, respectively. Then by~(2), there is a directed $X$-path $Q$ from $v$ to $u$ satisfying that $a_f|A(Q)\cap Z_1|+b_f|A(Q)\cap Z_2|\equiv r_f \pmod {q_f}$.
Let $H:=H'\cup Q$. Then $H$ is a subdivision of $F$ contained in $D$ where for every $e\in  A(F)$, the corresponding branching path $P_e$ of $H$ satisfies that $a_e|A(P_e)\cap Z_1|+b_e|A(P_e)\cap Z_2|\equiv r_e \pmod {q_e}$.
\end{proof}

We discuss that we can obtain an analogous result for undirected graphs. For an undirected graph $G$, we denote by $E(G)$ the set of edges in $G$. For an undirected graph $G$, we denoted by $\lrG$ the \emph{biorientation} of $G$, which is the digraph obtained by replacing each edge $uv$ of $G$ with two arcs $(u,v)$ and $(v,u)$. 

Let $G$ be an undirected graph and let $B_1, B_2$ be two sets of edges of $G$.
A cycle $C$ of $G$ is called \emph{$(B_1, B_2)$-unbalanced} if $|E(C)\cap B_1|\neq |E(C)\cap B_2|$.
Let $\mu^*(G, B_1, B_2)$ be the minimum number of parts in a vertex partition $\mathcal{P}$ of $G$ such that for every $P\in \mathcal{P}$, the subgraph induced by $P$ of $G$ contains no $(B_1, B_2)$-unbalanced directed cycles. 

\begin{corollary}\label{cor:undirected}
Let $U$ be an undirected graph, and for every $e\in E(U)$ let $(a_e,b_e,r_e,q_e)$ be a tuple of integers with $q_e \ge 2$ and $\gcd(a_e,q_e)=\gcd(b_e,q_e)=1$.
Then there exists an integer $M$ satisfying the following. 

Let $G$ be an undirected graph and $B_1,B_2\subseteq E(G)$. If $\mu^*(G,B_1,B_2)
\ge M$, then $G$ contains a
subdivision of $U$ such that for every $e\in  E(U)$, the corresponding branching path $P_e$ in $H$ satisfies that $a_e|E(P_e)\cap B_1|+b_e|E(P_e)\cap B_2|\equiv r_e \pmod {q_e}$.
\end{corollary}
\begin{proof}
Let $\lrU$ be the biorientation of $U$. For each $e=uv\in E(U)$, let $(u,v)$ and $(v,u)$ be the corresponding arcs contained in $\lrU$ where both $(u,v)$ and $(v,u)$ are assigned with $a_e$, $b_e$, $r_e$ and $q_e$. Then by Theorem~\ref{thm:mainthm}, there is an integer $N$ such that for any digraph $D$ with $Z_1,Z_2\subseteq A(D)$ satisfying that $\mu(D,Z_1,Z_2)\ge N$, $D$ contains a
subdivision $H$ of $\lrU$ such that for every $e\in  A(\lrU)$, the corresponding  branching path $P_e$ in $H$ satisfies that $a_e|A(P_e)\cap Z_1|+b_e|A(P_e)\cap Z_2|\equiv r_e \pmod {q_e}$.

Set $M:=N$. Let $G$ and $B_1,B_2$ be given such that $\mu^*(G,B_1,B_2)\ge N$. Let $\lrG$ be the biorientation of $G$, and for each $i\in [2]$, let $B_i^*:=\{(u,v), (v,u):uv\in B_i\}$ be the set of all arcs corresponding to edges in $B_i$.

Observe that every partition of $V(\lrG)$ where each part induces a subdigraph of $\lrG$ with no $(B_1^*,B_2^*)$-unbalanced directed cycles is also a partition of $V(G)$ where each part induces a subgraph of $G$ with no $(B_1,B_2)$-unbalanced cycles.  It implies that $\mu(\lrG,B_1^*,B_2^*)\ge \mu^*(G,B_1,B_2)\ge N$.
Thus $\lrG$ contains a
subdivision $H$ of $\lrU$ such that for every $e\in  A(\lrU)$, the corresponding branching path $P_e$ in $H$ satisfies that $a_e|A(P_e)\cap B_1^*|+b_e|A(P_e)\cap B_2^*|\equiv r_e \pmod {q_e}$. Therefore, $G$ contains a subdivision $H^*$ of $U$ inherited from $H$ where for each $e\in E(U)$, the corresponding branching path $P_e$ in $H^*$ satisfies that $a_e|A(P_e)\cap B_1|+b_e|A(P_e)\cap B_2|\equiv r_e \pmod {q_e}$.  
\end{proof}


\begin{thebibliography}{10}

\bibitem{Aboulker2019}
P.~Aboulker, N.~Cohen, F.~Havet, W.~Lochet, P.S. Moura, and S.~Thomass\'{e}.
\newblock Subdivisions in digraphs of large out-degree or large dichromatic
  number.
\newblock {\em Electronic Journal of Combinatorics}, 26(3):3--19, 2019.

\bibitem{AboulkerHKR2022}
Pierre Aboulker, {Fr\'{e}d\'{e}ric} Havet, Kolja Knauer, and {Cl\'{e}ment}
  Rambaud.
\newblock {On the dichromatic number of surfaces}.
\newblock {\em Electronic Journal of Combinatorics}, 29(1):1--30, 2022.

\bibitem{BollobasT1998}
B\'{e}la Bollob\'{a}s and Andrew Thomason.
\newblock {Proof of a conjecture of {M}ader, {E}rd\H{o}s and {H}ajnal on
  topological complete subgraphs}.
\newblock {\em European J. Combin.}, 19(8):883--887, 1998.

\bibitem{ChitnisCHM2015}
Rajesh~Hemant Chitnis, Marek Cygan, Mohammad~Taghi Hajiaghayi, and D\'{a}niel
  Marx.
\newblock {Directed subset feedback vertex set is fixed-parameter tractable}.
\newblock {\em ACM Transactions on Algorithms}, 11(4):1--28, 2015.

\bibitem{CyganPPO2013}
Marek Cygan, Marcin Pilipczuk, Micha{\l} Pilipczuk, and Jakub~Onufry
  Wojtaszczyk.
\newblock {Subset feedback vertex set is fixed-parameter tractable}.
\newblock {\em SIAM Journal on Discrete Mathematics}, 27(1):290--309, 2013.

\bibitem{EvenNZ2000}
Guy Even, Joseph Naor, and Leonid Zosin.
\newblock {An 8-approximation algorithm for the subset feedback vertex set
  problem}.
\newblock {\em SIAM Journal on Computing}, 30(4):1231--1252, 2000.

\bibitem{GiraoPS2021}
Ant\'{o}nio Gir\~{a}o, Kamil Popielarz, and Richard Snyder.
\newblock Subdivisions of digraphs in tournaments.
\newblock {\em J. Combin. Theory Ser. B}, 146:266--285, 2021.

\bibitem{GishbolinerSS2022}
Lior Gishboliner, Raphael Steiner, and Tibor {Szab\'{o}}.
\newblock {Dichromatic number and forced subdivisions}.
\newblock {\em J. Combin. Theory Ser. B}, 153(C):1--30, 2022.

\bibitem{GollinHKKO2021}
J.~Pascal Gollin, Kevin Hendrey, Ken{-}ichi Kawarabayashi, O{-}joung Kwon, and
  Sang{-}il Oum.
\newblock {A unified half-integral Erd\H{o}s-P\'osa theorem for cycles in
  graphs labelled by multiple abelian groups}.
\newblock arXiv:2102.01986, 2021.

\bibitem{GollinHKOY2022}
J.~Pascal Gollin, Kevin Hendrey, O{-}joung Kwon, Sang{-}il Oum, and Youngho
  Yoo.
\newblock {A unified Erd\H{o}s-P\'osa theorem for cycles in graphs labelled by
  multiple abelian groups}.
\newblock arXiv:2209.09488, 2022.

\bibitem{HolsK2018}
{Eva-Maria}~C Hols and Stefan Kratsch.
\newblock {A randomized polynomial kernel for subset feedback vertex set}.
\newblock {\em Theory of Computing Systems}, 62(1):63--92, 2018.

\bibitem{HuynhJW2017}
Tony Huynh, Felix Joos, and Paul Wollan.
\newblock {A unified Erd\H{o}s-P\'{o}sa theorem for constrained cycle}s.
\newblock {\em Combinatorica}, 39(1):91--133, 2019.

\bibitem{KakimuraK2012b}
Naonori Kakimura and Ken{-}ichi Kawarabayashi.
\newblock Packing cycles through prescribed vertices under modularity
  constraints.
\newblock {\em Adv. in Appl. Math.}, 49(2):97--110, 2012.

\bibitem{KomlosS1996}
J\'{a}nos Koml\'{o}s and Endre Szemer\'{e}di.
\newblock Topological cliques in graphs. {II}.
\newblock {\em Combin. Probab. Comput.}, 5(1):79--90, 1996.

\bibitem{TamasS2022}
{Tam\'{a}s} {M\'{e}sz\'{a}ros} and Raphael Steiner.
\newblock Complete directed minors and chromatic number.
\newblock {\em Journal of Graph Theory}, 101(4):623--632, 2022.

\bibitem{MoharW2016}
Bojan Mohar and Hehui Wu.
\newblock Dichromatic number and fractional chromatic number.
\newblock {\em Forum of Mathematics, Sigma}, 4:{E32}, 2016.

\bibitem{Neumann1982}
V.~Neumann-Lara.
\newblock The dichromatic number of a digraph.
\newblock {\em J. Combin. Theory Ser. B}, 33(3):265--270, 1982.

\bibitem{PontecorviW2012}
M.~Pontecorvi and P~Wollan.
\newblock Disjoint cycles intersecting a set of vertices.
\newblock {\em J. Combin. Theory Ser. B}, 102(5):1134--1141, 2012.

\bibitem{Steiner2022}
Raphael Steiner.
\newblock {Subdivisions with congruence constraints in digraphs of large
  chromatic number}.
\newblock arXiv:2208.06358, 2022.

\bibitem{Thomassen1983}
Carsten Thomassen.
\newblock Graph decomposition with applications to subdivisions and path
  systems modulo {$k$}.
\newblock {\em J. Graph Theory}, 7(2):261--271, 1983.

\bibitem{Thomassen1985}
Carsten Thomassen.
\newblock Even cycles in directed graphs.
\newblock {\em European J. Combin.}, 6(1):85--89, 1985.

\bibitem{Thomassen2001}
Carsten Thomassen.
\newblock Totally odd {$K_4$}-subdivisions in 4-chromatic graphs.
\newblock {\em Combinatorica}, 21(3):417--443, 2001.

\bibitem{Zang1998}
Wenan Zang.
\newblock {Proof of Toft’s conjecture: Every graph containing no fully odd
  $K_4$ is 3-colorable}.
\newblock {\em Journal of Combinatorial Optimization}, 2:117--188, 1998.

\end{thebibliography}
\end{document}